\documentclass[11pt]{article}
\usepackage{amsmath,amsthm,amssymb,xypic}

\theoremstyle{plain}  

\newtheorem{thm}{Theorem}[section]

\newtheorem{lem}[thm]{Lemma}   
\newtheorem{cor}[thm]{Corollary}

\theoremstyle{definition}

\newtheorem{prop-defn}[thm]{Proposition--Definition}
\newtheorem{defn}[thm]{Definition}

\theoremstyle{remark}

\newtheorem{rem}[thm]{Remark}
\newtheorem{ex}[thm]{Example}
\newtheorem{con}[thm]{Construction}

\DeclareMathOperator{\Hom}{Hom}

\DeclareMathOperator{\Hilb}{Hilb}

\DeclareMathOperator{\PGL}{PGL}

\newcommand{\QED}{\ifhmode\unskip\nobreak\fi\quad {\rm Q.E.D.}} 

\newcommand{\bC}{\mathbb C}

\newcommand{\bH}{\mathbb H}

\newcommand{\bP}{\mathbb P}
\newcommand{\bQ}{\mathbb Q}
\newcommand{\bR}{\mathbb R}

\newcommand{\bZ}{\mathbb Z}

\newcommand{\cA}{\mathcal A}

\newcommand{\cG}{\mathcal G}

\newcommand{\cO}{\mathcal O}

\newcommand{\cU}{\mathcal U}

\newcommand{\oM}{\overline{M}}

\DeclareMathOperator{\ord}{ord}
\DeclareMathOperator{\val}{val}
\newcommand{\oX}{\overline{X}}
\newcommand{\oY}{\overline{Y}}
\newcommand{\oZ}{\overline{Z}}
\DeclareMathOperator{\Gr}{Gr}

\title{The homology of tropical varieties}
\author{Paul Hacking}
\begin{document}
\maketitle
\section{Introduction}
Given a closed subvariety $X$ of an algebraic torus $T$, the associated tropical variety is a polyhedral fan in the space of $1$-parameter subgroups of the torus which describes the behaviour of the subvariety at infinity.
We show that the link of the origin has only top rational homology if a genericity condition is satisfied.
Our result is obtained using work of Tevelev \cite{T} and Deligne's theory of mixed Hodge structures \cite{D}.

Here is a sketch of the proof. We use the tropical variety of $X$ to construct a smooth compactification $X \subset \oX$ with simple normal crossing boundary $B$. We relate the link $L$ of the tropical variety to the \emph{dual complex} $K$ of $B$, that is, the simplicial complex with vertices corresponding to the irreducible components $B_i$ of $B$ and simplices of dimension $j$ corresponding to $(j+1)$-fold intersections of the $B_i$.
Following \cite{D} we identify the homology groups of $K$ with graded pieces of the weight filtration of the cohomology of $X$.
Since $X$ is an affine variety, it has the homotopy type of a CW complex of real dimension equal to the complex dimension of $X$.
From this we deduce that $K$ and $L$ have only top homology. 

The link of the tropical variety of $X \subset T$ was previously shown to have only top homology in the following cases: 
the intersection of the Grassmannian $G(3,6)$ with the big torus $T$ in its Pl\"ucker embedding \cite{SS}, 
the complement of an arrangement of hyperplanes \cite{AK}, 
and the space of matrices of rank $\le 2$ in $T=(\bC^{\times})^{m \times n}$ \cite{MY}.
We discuss these and other examples from our viewpoint in Sec.~\ref{secex}.

It has been conjectured that the link of the tropical variety of an \emph{arbitrary} subvariety of a torus is homotopy equivalent to a bouquet of spheres (so, in particular, has only top homology).
I expect that this is false in general, but I do not know a counterexample. See also Rem.~\ref{conj}.

We note that D.~Speyer has used similar techniques to study the topology of the tropicalisation of a curve defined over the field $\bC((t))$ of formal power series, see \cite[Sec.~10]{S}.

\section{Statement of Theorem}

We work throughout over $k=\bC$.
Let $X \subset T$ be a closed subvariety of an algebraic torus $T \simeq (\bC^{\times})^r$. 
Let $K=\bigcup_{n \ge 1} \bC((t^{1/n}))$ be the field of Puiseux series 
(the algebraic closure of the field $\bC((t))$ of Laurent series) and 
$\ord \colon K^{\times} \rightarrow \bQ$ the valuation of $K/\bC$ such that $\ord(t)=1$.

Let $M=\Hom(T,\bC^{\times}) \simeq \bZ^r$ be the group of characters of $T$ and $N=M^*$. 
We have a natural map
$$\val \colon T(K) \rightarrow N_{\bQ}$$
given by 
$$T(K) \ni P \mapsto (\, \chi^m \mapsto \ord(\chi^m(P)) \,).$$
In coordinates
$$(K^{\times})^r \ni (a_1,\ldots,a_r) \mapsto (\ord(a_1),\ldots,\ord(a_r)) \in \bQ^r$$

\begin{defn}\cite[1.2.1]{EKL}
The \emph{tropical variety} $\cA$ of $X$ is the closure of $\val(X(K))$ in $N_{\bR} \simeq \bR^r$.
\end{defn}

\begin{thm}\cite[2.2.5]{EKL} \label{dimension}
$\cA$ is the support of a rational polyhedral fan in $N_{\bR}$ of pure dimension $\dim X$.
\end{thm}

Let $\Sigma$ be a rational polyhedral fan in $N_{\bR}$.
Let $T \subset Y$ be the associated torus embedding.
Let $\oX=\oX(\Sigma)$ be the closure of $X$ in $Y$.

\begin{thm}\cite[2.3]{T}\label{compact}
$\oX$ is compact iff the support $|\Sigma|$ of $\Sigma$ contains $\cA$.
\end{thm}

From now on we always assume that $\oX$ is compact.

\begin{thm}\cite[3.9]{ST}\cite{T2} \label{expdim}
The intersection $\oX \cap O$ is non-empty and has pure dimension equal to the expected dimension for every torus orbit 
$O \subset Y$ iff $|\Sigma|=\cA$.
\end{thm}

\begin{proof}
Suppose $|\Sigma|=\cA$.
We first show that $\oX \cap O$ is nonempty for every orbit $O \subset Y$.
Let $\Sigma' \rightarrow \Sigma$ be a strictly simplical refinement of $\Sigma$ and 
$f \colon Y' \rightarrow Y$ the corresponding toric resolution of $Y$.
Let $\oX'$ be the closure of $X$ in $Y'$.
Let $O \subset Y$ be an orbit, and $O' \subset Y'$ an orbit such that $f(O') \subseteq O$.
Then $\oX' \cap O' \neq \emptyset$ by \cite[2.2]{T}, and $f(\oX' \cap O') \subseteq \oX \cap O$,
so $\oX \cap O \neq \emptyset$ as required.

We next show that $\oX \cap O$ has pure dimension equal to the expected dimension for every orbit $O \subset Y$.
Let $O \subset Y$ be an orbit of codimension $c$.
Let $Z$ be an irreducible component of the intersection $\oX \cap O$ with its reduced induced structure.
Let $W$ be the closure of $O$ in $Y$ and $\oZ$ the closure of $Z$ in $W$.
Then, since $\oZ$ is compact, the fan of the toric variety $W$ contains the tropical variety of 
$Z \subset O$ by Thm.~\ref{compact}. We deduce that $\dim Z \le \dim X - c$ by Thm.~\ref{dimension}.
On the other hand, since toric varieties are Cohen-Macaulay, the orbit $O \subset Y$ is cut out set-theoretically
by a regular sequence of length $c$ at each point of $O$.
It follows that $\dim Z \ge \dim X - c$, so $\dim Z = \dim X - c$ as required.

The converse follows from \cite[3.9]{ST}.
\end{proof}

Here is the main result of this paper.

\begin{thm}\label{mainthm} Suppose that $|\Sigma|=\cA$ and the following condition is satisfied:
\begin{enumerate}
\item[$(*)$] For each torus orbit $O \subset Y$, $\oX \cap O$ is smooth and is connected if it has positive dimension. 
\end{enumerate}
Then the link $L$ of $0 \in \cA$ has only top reduced rational homology, i.e.,
$\tilde{H}_i(L,\bQ)=0$ for $i < \dim L = \dim X -1$.
\end{thm}

\begin{ex} \label{exci}
Let $\oY$ be a projective toric variety.
Let $\oX \subset \oY$ be a complete intersection.
That is, $\oX= H_1 \cap \cdots \cap H_c$ where $H_i$ is an ample divisor on $\oY$.
Assume that $H_i$ is a general element of a basepoint free linear system for each $i$.
Let $Y \subset \oY$ be the open toric subvariety consisting of orbits meeting $\oX$ and $\Sigma$ the fan of $Y$.
Then $|\Sigma|=\cA$ by Thm.~\ref{expdim} and $\oX \subset Y$ satisifes the condition $(*)$ by Bertini's theorem \cite[III.7.9, III.10.9]{H}.

If $\overline{\Sigma}$ is the (complete) fan of $\oY$, the fan $\Sigma$ is the union of the cones of $\overline{\Sigma}$ of codimension $\ge c$.
So it is clear in this example that the link $L$ of $0 \in \cA$ has only top reduced homology.
Indeed, let $r= \dim Y$. Then the link $K$ of $0 \in \overline{\Sigma}$ is a polyhedral subdivision of the $(r-1)$-sphere, and
$L$ is the $(r-c-1)$-skeleton of $K$, hence $\tilde{H}_i(L,\bZ)=\tilde{H}_i(S^{r-1},\bZ)=0$ for $i<r-c-1$. 
\end{ex}

A useful reformulation of condition $(*)$ is given by the following lemma. 

\begin{lem} \label{smoothness}
Assume that $|\Sigma|=\cA$.
Then the following conditions are equivalent.
\begin{enumerate}
\item $\oX \cap O$ is smooth for each orbit $O \subset Y$.
\item The multiplication map $m \colon T \times \oX \rightarrow Y$ is smooth.
\end{enumerate}
\end{lem}
\begin{proof}
The fibre of the multiplication map over a point $y \in O \subset Y$ is isomorphic to $(\oX \cap O) \times S$,
where $S \subset T$ is the stabiliser of $y$.
Now $m$ is smooth iff it is flat and each fibre is smooth.
The map $m$ is surjective and has equidimensional fibres by Thm.~2.4.
Finally, if $W$ is integral, $Z$ is normal, and $f \colon W \rightarrow Z$ is dominant and has reduced fibres, then $f$ is flat iff it 
has equidimensional fibres by \cite[14.4.4, 15.2.3]{EGA4}.
This gives the equivalence.
\end{proof}

\begin{defn}\cite[1.1,1.3]{T}
We say $\oX \subset Y$ is \emph{tropical} if $m \colon T \times \oX \rightarrow Y$ is flat and surjective.
(Then in particular $\oX \cap O$ is non-empty and has the expected dimension for each orbit $O \subset Y$, so $|\Sigma|=\cA$ by Thm.~\ref{expdim}.)
We say $X \subset T$ is \emph{sch\"on} if $m$ is smooth for some (equivalently, any \cite[1.4]{T}) tropical compactification $\oX \subset Y$.
\end{defn}

\begin{ex}
Here we give some examples of sch\"on subvarieties of tori. (For more examples see Sec.~\ref{secex}.)
\begin{enumerate}
\item Let $\oY$ be a projective toric variety and $\oX \subset \oY$ a general complete intersection as
in Ex.~\ref{exci}. Let $T \subset \oY$ be the big torus and $X = \oX \cap T$.
Then $\oX \cap O$ is either empty or smooth of the expected dimension for every orbit $O \subset \oY$ by Bertini's theorem. Hence $X \subset T$ is sch\"on.
\item Let $\oY$ be a projective toric variety and $G$ a group acting transitively on $\oY$.
Let $\oX \subset \oY$ be a smooth subvariety. Then, for $g \in G$ general, $g\oX \cap O$ is either empty or smooth of the expected dimension for every orbit $O \subset Y$ by \cite[III.10.8]{H}.
Let $T \subset \oY$ be the big torus and $X' = g\oX \cap T$. Then $X' \subset T$ is sch\"on for $g \in G$ general.
\end{enumerate}
\end{ex}

\begin{ex} \label{notschon}
Here is a simple example $X \subset T$ which is not sch\"on.
Let $\oY$ be a projective toric variety and $\oX \subset \oY$ a closed subvariety such that 
$\oX$ meets the big torus $T \subset \oY$ and $\oX$ is singular at a point which is contained 
in an orbit $O \subset \oY$ of codimension $1$.
Let $X = \oX \cap T$. Then $X \subset T$ is not sch\"on.
Indeed, suppose that $m \colon T \times \oX' \rightarrow Y'$ is smooth for some tropical compactification $\oX' \subset Y'$. We may assume that the toric birational map 
$f \colon Y' \dashrightarrow \oY$ is a morphism by \cite[2.5]{T}.
Now $\oX \cap O$ is singular by construction, and $f \colon Y' \rightarrow \oY$ is an isomorphism over $O$ because $O \subset \oY$ has codimension $1$, hence $\oX' \cap f^{-1}O$ is also singular, a contradiction. 
\end{ex}

\begin{rem} \label{conj}
It has been suggested that the link $L$ of the tropical variety of an \emph{arbitrary} subvariety of a torus is homotopy equivalent to a bouquet of top dimensional spheres (so, in particular, has only top homology). I expect that this is false in general, but I do not know a counterexample.
However, there are many examples where the hypothesis $(*)$ of Thm.~\ref{mainthm} is not satisfied but the conclusion is still valid. For example, let $\oX \subset \oY$ be a complete intersection in a projective toric variety such that $\oX \cap O$ has the expected dimension for each orbit 
$O \subset \oY$ and let $X= \oX \cap T \subset T$ where $T \subset \oY$ is the big torus. 
Then $X \subset T$ is not sch\"on in general but $L$ is a bouquet of top-dimensional spheres, cf. Ex.~\ref{notschon}, \ref{exci}. See also Ex.~4.4 for another example.
\end{rem}

\begin{con} \label{tropconstruct}\cite[1.7]{T}
We can always construct a tropical compactification $\oX \subset Y$ as follows.
Choose a projective toric compactification $\oY_0$ of $T$.
Let $\oX_0$ denote the closure of $X$ in $\oY_0$.
Assume for simplicity that $$S=\{ t \in T \ | \ t \cdot X =X \} \subset T$$ is trivial
(otherwise, we can pass to the quotient $X/S \subset T/S$).
Consider the embedding 
$T \hookrightarrow \Hilb(\oY_0)$ given by $t \mapsto t^{-1}[\oX_0]$.
Let $\oY$ be the normalisation of the closure of $T$ in $\Hilb(\oY_0)$.
(So $\oY$ is a projective toric compactification of $T$.)
Let $\oX$ be the closure of $X$ in $\oY$, and $Y \subset \oY$ the open toric subvariety consisting of orbits meeting $\oX$.
Let $\cU \subset \Hilb(\oY_0) \times \oY_0$ denote the universal family over $\Hilb(\oY_0)$ and 
$\cU^0 =\cU \cap (\Hilb(\oY_0) \times T)$.
One shows that there is an identification
\begin{equation} \label{multuniv}
\xymatrix{
T \times X \ar[rr]^{\sim} \ar[rd]_m & & \cU^0|_Y \ar[ld]\\
& Y}
\end{equation}
given by $(t,x) \mapsto (tx,t)$ \cite[p.~1093, Pf. of 1.7]{T}. In particular, $m$ is flat.
\end{con}

\begin{rem}
We note that, in the situation of \ref{tropconstruct}, we can verify the condition $(*)$ using Gr\"obner basis techniques.
Let $O \subset Y$ be an orbit.
Let $\sigma$ be the cone in the fan of $Y$ corresponding to $O$, and $w \in N$ an integral point in the relative interior of $\sigma$.
We regard $w$ as a 1-parameter subgroup $\bC^{\times} \rightarrow T$ of $T$.
Then, by construction, the limit $\lim_{t \rightarrow 0} w(t)$ lies in the orbit $O$. 
Let $\oX_0^w$ be the flat limit of the 1-parameter family $w(t)^{-1}\oX_0$ as $t \rightarrow 0$.
Then the fibres of $\cU \rightarrow \Hilb(\oY_0)$ over $O$ are the translates of $\oX_0^w$.
Let $y \in O$ be a point and $S \subset T$ the stabiliser of $y$.
The fibre of $m$ over $y$ is isomorphic to both $(\oX \cap O) \times S$ and $\oX_0^w \cap T$ (by the identification (\ref{multuniv})).
Hence $\oX \cap O$ is smooth (resp. connected) iff $\oX_0^w \cap T$ is so.
Suppose now that $\oY_0 \simeq \bP^N$, and let $I \subset k[X_0,\ldots,X_N]$ be the homogeneous ideal of $\oX_0 \subset \bP^N$.
Then $\oX_0^w$ is the zero locus of the initial ideal of $I$ with respect to $w$.
\end{rem}

\section{The stratification of the boundary and the weight filtration}

Let $\oX$ be a smooth projective variety of dimension $n$, and $B \subset \oX$ a simple normal crossing divisor.
We define the \emph{dual complex} of $B$ to be the CW complex $K$ defined as follows.
Let $B_1, \ldots, B_m$ be the irreducible components of $B$ and write $B_I = \bigcap_{i \in I } B_i$ for $I \subset [m]$.
To each connected component $Z$ of $B_I$ we associate a simplex $\sigma$ with vertices labelled by $I$.
The facet of $\sigma$ labelled by $I\setminus\{i\}$ is identified with the simplex corresponding to the connected component of 
$B_{I \setminus \{ i \}}$ containing $Z$.

\begin{thm} \label{homK}
The reduced homology of $K$ is identified with the top graded pieces of the weight filtration on the cohomology of 
the complement $X = \oX \backslash B$. Precisely, 
$$\tilde{H}_{i}(K,\bC)=\Gr^W_{2n} H^{2n-(i+1)}(X,\bC).$$
\end{thm}

\begin{cor} \label{affine}
If $X$ is affine, then
$$\tilde{H}_{i}(K,\bC)=\left\{ \begin{array}{cc}
			\Gr^W_{2n} H^{n}(X,\bC) & \mbox{ if } i=n-1 \\
			0                      & \mbox{otherwise}\\
			\end{array} \right.$$
\end{cor}

\begin{proof}[Proof of Thm.~\ref{homK}]
This is essentially contained in \cite{D}, see also \cite[Sec.~8.4]{V}.
Define a filtration $\tilde{W}$ of the complex $\Omega_{\oX}^{\cdot}(\log B)$ of differential forms on ${\oX}$ with logarithmic 
poles along $B$ by
$$\tilde{W}_l \Omega^k_{\oX}(\log B)= \Omega^l_{\oX}(\log B) \wedge \Omega_{\oX}^{k-l}.$$
The filtration of  $\Omega_{\oX}^{\cdot}(\log B)$ yields a spectral sequence 
$$E_1^{p,q}=\bH^{p+q}({\oX},\Gr^{\tilde{W}}_{-p}\Omega^{\cdot}_{\oX}(\log B)) \implies \bH^{p+q}(\Omega^{\cdot}_{\oX}(\log B))=H^{p+q}(X,\bC).$$
which defines a filtration $\tilde{W}$ on $H^{\cdot}(X,\bC)$. 
The \emph{weight filtration} $W$ on $H^i(X,\bC)$ is by definition the shift $W=\tilde{W}[i]$, i.e.,
$W_j(H^i)=\tilde{W}_{j-i}(H^i)$.
The spectral sequence degenerates at $E_2$ 
\cite[3.2.10]{D}, so 
$$E_2^{p,q} = \Gr^{\tilde{W}}_{-p} H^{p+q}(X,\bC).$$
The $E_1$ term may be computed as follows. 
Let $\tilde{B}^l$ denote the disjoint union of the $l$-fold intersections of the components of $B$, 
and $j_l$ the map $\tilde{B^l} \rightarrow {\oX}$. (By convention $\tilde{B}^0={\oX}$.)
The Poincar\'{e} residue map defines an isomorphism
\begin{eqnarray}
\Gr^{\tilde{W}}_l \Omega^k_{\oX}(\log B) \stackrel{\sim}{\longrightarrow} {j_l}_{*} \Omega_{\tilde{B^l}}^{k-l},
\end{eqnarray}
see \cite[Prop.~8.32]{V}.
This gives an identification
$$E_1^{p,q}=\bH^{p+q}({\oX},\Gr^{\tilde{W}}_{-p} \Omega_{\oX}^{\cdot}(\log B)) 
= \bH^{2p+q}(\tilde{B}^{(-p)}, \Omega^{\cdot}_{\tilde{B}^{(-p)}}) = H^{2p+q}(\tilde{B}^{(-p)},\bC).$$
The differential
$$d_1 \colon H^{2p+q}(\tilde{B}^{(-p)}) \rightarrow H^{2(p+1)+q}(\tilde{B}^{(-p-1)})$$
is identified (up to sign) with the Gysin map on components \cite[Prop.~8.34]{V}.
Precisely, write $s=-p$.
Then $d_1 \colon H^{q-2s}(\tilde{B}^{(s)}) \rightarrow H^{q-2(s-1)}(\tilde{B}^{(s-1)})$ is given by the maps
$$(-1)^{s+t}j_{*} \colon H^{q-2s}(B_{I}) \rightarrow H^{q-2(s-1)}(B_{J}),$$
where $I=\{i_1<\cdots<i_s\}$, $J=I \backslash \{i_t\}$, $j$ denotes the inclusion $B_{I} \subset B_{J}$, and 
$j_*$ is the Gysin map.
Equivalently, identify $H^{q-2s}(\tilde{B}^{(s)})=H_{2n-q}(\tilde{B}^{(s)})$ by Poincar\'{e} duality. Then
$d_1 \colon H_{2n-q}(\tilde{B}^{(s)}) \rightarrow H_{2n-q}(\tilde{B}^{(s-1)})$ is given by the maps
$$(-1)^{s+t}j_* \colon H_{2n-q}(\tilde{B}^{(s)}) \rightarrow H_{2n-q}(\tilde{B}^{(s-1)}).$$
So, the $E_1$ term of the spectral sequence is as follows.
\begin{eqnarray*}
\renewcommand{\arraystretch}{1.5}
\begin{array}{cccccccccc}
H_0(\tilde{B}^{(n)}) & \rightarrow & H_0(\tilde{B}^{(n-1)}) & \rightarrow & \cdots & \rightarrow & H_0(\tilde{B}^{(1)}) & \rightarrow & H_0(\tilde{B}^{(0)}) \\
                     &             & H_1(\tilde{B}^{(n-1)}) & \rightarrow & \cdots & \rightarrow & H_1(\tilde{B}^{(1)}) & \rightarrow & H_1(\tilde{B}^{(0)}) \\
                     &             &                        &             &        &             & \vdots               &             & \vdots \\
                     &             &                        &             &        &             &
H_{2n-2}(\tilde{B}^{(1)}) & \rightarrow & H_{2n-2}(\tilde{B}^{(0)}) \\                    
                     &             &                        &             &        &             &                     &             & H_{2n-1}(\tilde{B}^{(0)}) \\
                     &             &                        &             &        &             &                     &             & H_{2n}(\tilde{B}^{(0)}) \\
\end{array}
\end{eqnarray*}
The top row ($q=2n$) is the complex 
$$ \cdots \rightarrow H_{0}(\tilde{B}^{(s+1)}) \rightarrow H_{0}(\tilde{B}^{(s)}) \rightarrow H_{0}(\tilde{B}^{(s-1)}) \rightarrow \cdots, $$
which computes the reduced homology of the dual complex $K$ of $B$.
We deduce
$$\Gr_{s}^{\tilde{W}}H^{2n-s}(X,\bC) = \tilde{H}_{s-1}(K,\bC).$$
\end{proof}

\begin{proof}[Proof of Corollary~\ref{affine}]
If $X$ is affine then $X$ has the homotopy type of a CW complex of dimension $n$, so $H^k(X,\bC)=0$ for $k > n$.
\end{proof}

\begin{proof}[Proof of Thm.~\ref{mainthm}]
By our assumption and Lem.~\ref{smoothness} the multiplication map $m \colon T \times \oX \rightarrow Y$ is smooth.
Let $Y' \rightarrow Y$ be a toric resolution of $Y$ given by a refinement $\Sigma'$ of $\Sigma$.
Then $m' \colon T \times \oX' \rightarrow Y'$ is also smooth --- it is the pullback of $m$ 
\cite[2.5]{T}.
So $\oX'$ is smooth with simple normal crossing boundary $B'=\oX'\setminus X$
(because this is true for $Y'$).
Hence the dual complex $K$ of $B'$ has only top reduced rational homology by Cor.~\ref{affine}.

It remains to relate $K$ and the link $L$ of $0 \in \cA$.
Recall that the fan $\Sigma$ of $Y$ has support $\cA$.
The cones of $\Sigma$ of dimension $p$ correspond to toric strata $Z \subset Y$ of codimension 
$p$. These correspond to strata $Z \cap \oX \subset \oX$ of codimension $p$, which are connected
(by our assumption) unless $p = \dim \oX$. We can now construct $K$ from $L$ as follows.
Give $L$ the structure of a polyhedral complex induced by the fan $\Sigma$.
For each top dimensional cell, let $Z \subset Y$ be the corresponding toric stratum, and 
$k=|Z \cap \oX|$. We replace the cell by $k$ copies, identified along their boundaries.
Let $\hat{L}$ denote the resulting CW complex. Note immediately that $\hat{L}$ is homotopy equivalent to the one point union of $L$ and a collection
of top dimensional spheres.
So $\hat{L}$ has only top reduced rational homology iff $L$ does.
Finally let $\hat{L}'$ denote the subdivision of $\hat{L}$ induced by the refinement $\Sigma'$
of $\Sigma$. Then $\hat{L}'$ is the dual complex $K$ of $B'$. This completes the proof.
\end{proof}

We note the following corollary of the proof.

\begin{cor}\label{connected}
In the situation of Thm.~\ref{mainthm},
if in addition $\oX \cap O$ is connected for every orbit $O \subset Y$, 
then we have an identification $$\tilde{H}_{n-1}(L,\bC)=\Gr^W_{2n}H^n(X,\bC).$$
\end{cor}

\section{Examples} \label{secex}

We say a variety $X$ is \emph{very affine} if it admits a closed embedding in an algebraic torus.
If $X$ is very affine, the \emph{intrinsic torus} of $X$ is the torus $T$ with character lattice
$M=H^0(\cO_X^{\times})/k^{\times}$.
Choosing a splitting of the exact sequence
$$0 \rightarrow k^{\times} \rightarrow H^0(\cO_X^{\times}) \rightarrow M \rightarrow 0$$
defines an embedding $X \subset T$, and any two such are related by a translation.

\begin{ex} \label{hyperplanes}
Let $X$ be the complement of an arrangement of $m$ hyperplanes in $\bP^{n}$
whose stabiliser in $\PGL(n)$ is finite. 
Then $X$ is very affine with intrinsic torus $T= (\bC^{\times})^m /\bC^{\times}$, and the embedding $X \subset T$ is the restriction of the linear embedding $\bP^{n} \subset \bP^{m-1}$
given by the equations of the hyperplanes.
The embedding $X \subset T$ is sch\"on, and a tropical compactification 
$\oX \subset Y$ is given by Kapranov's visible contour construction, see \cite[Sec.~2]{HKT1}.
In \cite{AK} it was shown that the link $L$ of $0 \in \cA$ has only top reduced homology, and the rank of $H_{n-1}(L,\bZ)$ was computed using the M\"obius function of the lattice of flats of the matroid associated to the arrangement.
Thm.~\ref{mainthm} gives a different proof that the link has only top reduced rational homology.
Moreover, in this case $\oX \cap O$ is connected for every orbit $O \subset Y$, and the mixed Hodge structure on $H^{i}(X,\bC)$ is pure of weight $2i$ for each $i$. 
So we have an identification $$\tilde{H}_{n-1}(L,\bC)=\Gr^W_{2n} H^n(X,\bC)=H^n(X,\bC)$$
by Cor.~\ref{connected}.
\end{ex}

\begin{ex}
Let $X=M_{0,n}$, the moduli space of $n$ distinct points on $\bP^1$. The variety $X$ can be realised as the complement of a hyperplane arrangement in 
$\bP^{n-3}$, in particular it is very affine and the embedding $X \subset T$ in its intrinsic torus is sch\"on by Ex.~\ref{hyperplanes}.

More generally, consider the moduli space $X=X(r,n)$ of $n$ hyperplanes in linear general position in $\bP^{r-1}$. The Gel'fand--MacPherson correspondence identifies $X(r,n)$ 
with the quotient $G^0(r,n)/H$, where
$G^0(r,n) \subset G(r,n)$ is the open subset of the Grassmannian where all Pl\"ucker coordinates are nonzero and $H=(\bC^{\times})^n/\bC^{\times}$ is the maximal torus which acts freely on 
$G^0(r,n)$. See \cite[2.2.2]{GeM}. Thus the tropical variety $\cA$ of $X(r,n)$ is identified 
(up to a linear space factor) with the tropical Grassmannian $\cG(r,n)$ studied in \cite{SS}.
In particular, for $r=2$, the tropical variety of $M_{0,n}$ corresponds to $\cG(2,n)$, 
the so called space of phylogenetic trees. 
For $(r,n)=(3,6)$, the link $L$ of $0 \in \cA$ has only top reduced homology, and the top homology is free of rank $126$ \cite[5.4]{SS}.
Jointly with Keel and Tevelev, we showed that the embedding $X \subset T$ of $X(3,6)$ in its intrinsic torus is sch\"on (using work of Lafforgue \cite{L}) and described a tropical compactification $\oX \subset Y$ explicitly. So Thm.~\ref{mainthm} gives an alternative proof that
$L$ has only top reduced rational homology. Moreover, $\oX \cap O$
is connected for each orbit $O \subset Y$, and the mixed Hodge structure on $H^i(X(3,6),\bC)$ is pure of weight $2i$ for each $i$ by \cite[10.22]{HM}. So by Cor.~\ref{connected} we have an identification
$$H_{d-1}(L,\bC)=\Gr^W_{2d}H^d(X(3,6),\bC)=H^d(X(3,6),\bC)$$
where $d= \dim X(3,6)=4$. This agrees with the computation of $H^{\cdot}(X,\bC)$ in \cite{HM}.

We note that it is conjectured \cite[1.14]{KT} that $X(3,7)$ and $X(3,8)$ are sch\"on,
but in general the compactifications of $X(r,n)$ we obtain by toric methods will be highly singular by \cite[1.8]{L}. The cases $X(3,n)$ for $n \le 8$ are closely related to moduli spaces of
del Pezzo surfaces, see Ex.~\ref{dp} below
\end{ex}

\begin{ex} \cite{HKT2} \label{dp}
Let $X=X(n)$ denote the moduli space of smooth marked del Pezzo surfaces of degree $9-n$
for $4 \le n \le 8$.
Recall that a del Pezzo surface $S$ of degree $9-n$ is isomorphic to the blowup of $n$ points in $\bP^2$ which are in general position (i.e. no $2$ points coincide, no $3$ are collinear, no $6$ lie on a conic, etc).
A \emph{marking} of $S$ is an identification of the lattice $H^2(S,\bZ)$ with the standard lattice $\bZ^{1,n}$ of signature $(1,n)$ such that $K_S \mapsto -3e_0+e_1+\cdots+e_n$. It corresponds to a
realisation of $S$ as a blowup of $n$ ordered points in $\bP^2$.
Hence $X(n)$ is an open subvariety of $X(3,n)$ (because $X(3,n)$ is the moduli space of $n$
points in $\bP^2$ in \emph{linear} general position).
The lattice $K_S^{\perp} \subset H^2(X,\bZ)$ is isomorphic to the lattice $E_n$ (with negative definite intersection product). So the Weyl group $W=W(E_n)$ acts on $X(n)$ by changing the marking.
The action of the Weyl group $W$ on $X$ induces an action on the lattice $N$ of 
$1$-parameter subgroups of $T$ which preserves the tropical variety $\cA$ of $X$ in $N_{\bR}$. 
The link $L$ of $0 \in \cA$ is described in \cite[\S 7]{HKT2} in terms of sub root systems of $E_n$
for $n \le 7$. 

In \cite{HKT2} we showed that for $n \le 7$ the embedding $X \subset T$ of $X$ in its intrinsic torus is sch\"on  and described a tropical compactification $\oX \subset Y$ explicitly. The intersection $\oX \cap O$ is connected for each orbit $O \subset Y$.
So $L$ has only top reduced rational homology by Thm.~\ref{mainthm}, and 
$H_{d-1}(L,\bC)=\Gr^W_{2d}H^d(X(n),\bC)$ where $d= \dim X(n) = 2n-8$ by Cor.~\ref{connected}.
\end{ex}

\begin{ex}\cite{MY}
Let $\tilde{X} \subset (\bC^{\times})^{mn}$ be the space of matrices of size $m \times n$ and 
rank $\le 2$ with nonzero entries. (Thus $\tilde{X}$ is the zero locus of the $3 \times 3$ minors of the matrix.)
Let $X \subset T$ be the quotient of $\tilde{X} \subset (\bC^{\times})^{mn}$ by the torus
$(\bC^{\times})^m \times (\bC^{\times})^n$ acting by scaling rows and columns.
In \cite{MY} it was shown that the link $L$ of the origin in the tropical variety $\cA$ of 
$X \subset T$ is homotopy equivalent to a bouquet of top dimensional spheres.
Here we give an algebro-geometric interpretation of this result.

A point of $X$ corresponds to $n$ collinear points $\{p_i\}$ in the big torus in 
$\bP^{m-1}$, modulo simultaneous translation by the torus.
Let $f \colon X' \rightarrow X$ denote the space of lines through the points $\{p_i\}$.
The morphism $f$ is a resolution of $X$ with exceptional locus $\Gamma \simeq \bP^{m-2}$ over the singular point $P \in X$ where the $p_i$ all coincide.
Given a point $(C \subset \bP^{m-1},\{p_i\})$ of $X'$, 
let $q_j$ be the intersection of $C$ with the $j$th coordinate hyperplane. 
We obtain a pointed smooth rational curve $(C,\{p_i\},\{q_j\})$ such that $p_i \neq q_j$ for all 
$i$ and $j$, and the $q_j$ do not all coincide. 
Conversely, given such a pointed curve $(C,\{p_i\},\{q_j\})$, let $F_j$ be 
a linear form on $C \simeq \bP^1$ defining $q_j$. Then we obtain a linear embedding
$$F=(F_1:\cdots:F_m) \colon C \subset \bP^{m-1}$$
which is uniquely determined up to translation by the torus. 

We construct a compactification $X \subset \oX$ using a moduli space of pointed curves.
Let $\oX'$ denote the (fine) moduli space of pointed curves $(C,\{p_i\}_{1}^n,\{q_j\}_{1}^m)$ 
where $C$ is a proper connected nodal curve of arithmetic genus $0$ (a union of smooth rational curves such that the dual graph is a tree) and the $p_i$ and $q_j$ are smooth points of 
$C$ such that
\begin{enumerate}
\item $p_i \neq q_j$ for all $i$ and $j$. 
\item Each end component of $C$ contains at least one $p_i$ and one $q_j$, 
and each interior component of $C$ contains either a marked point or at least $3$ nodes. 
\item The $q_j$ do not all coincide.
\end{enumerate}
(The moduli space $\oX'$ can be obtained from $\oM_{0,n+m}$ as follows: for each boundary divisor 
$\Delta_{I_1,I_2}=\oM_{0,I_1 \cup \{*\}} \times \oM_{0,I_2 \cup \{*\}}$ we contract the $i$th factor to a point if $I_i \subseteq [1,n]$ or 
$I_i\subsetneq [n+1,n+m]$.)
Define the \emph{boundary} $B$ of $\oX'$ to be the locus where the curve $C$ is reducible.
It follows by deformation theory that $\oX'$ is smooth with normal crossing boundary $B$.
The construction of the previous paragraph defines an identification $X'= \oX' \setminus B$.
The desired compactification $X \subset \oX$ is obtained from $X' \subset \oX'$ 
by contracting $\Gamma \subset X'$.

Assume without loss of generality that $m \le n$.
Consider the resolution $f \colon X' \rightarrow X$ of $X$ with exceptional locus 
$\Gamma \simeq \bP^{m-2}$ described above. 
By \cite[Thm.~II.1.1*]{GoM} since $2\dim \Gamma \le \dim X$ and $X$ is affine it follows that $X'$ has the homotopy type of a CW complex of dimension $\dim X$.
Hence by Thm.~\ref{homK} the dual complex $K$ of the boundary $B$ has only top rational homology, and $\tilde{H}_{d-1}(K,\bC)=\Gr_{2d}H^d(X',\bC)$ where $d = \dim X' = m+n-3$.

The compactification $\oX$ of $X$ is a tropical compactification $\oX \subset Y$ of $X \subset T$ such that $\oX \cap O$ is connected for each orbit $O \subset Y$.
This is proved using the general result \cite[2.10]{HKT2}.
The toric variety $Y$ corresponds to the fan $\Sigma$ with support $\cA$ given by \cite[2.11]{MY}.
In particular, it follows that $K$ is a triangulation of the link $L$.
Hence we obtain an alternative proof that $L$ has only top reduced rational homology, and a
geometric interpretation of the top homology group.
\end{ex}

\medskip

\noindent
\textbf{Acknowledgements}:
I would like to thank J.~Tevelev for allowing me to include his unpublished result 
Thm.~\ref{expdim} and for many helpful comments.
I would also like to thank S.~Keel, S.~Payne, D.~Speyer, and B.~Sturmfels for useful discussions.
The author was partially supported by NSF grant DMS-0650052.

\medskip
\noindent
Paul Hacking,  Department of Mathematics, University of Washington, Box 354350, Seattle, WA~98195; 
\texttt{hacking@math.washington.edu} \\
\\


\begin{thebibliography}{99}
\bibitem[AK]{AK} F.~Ardila, C.~Klivans, The Bergman complex of a matroid and phylogenetic trees, 
J. Combin. Theory Ser. B~96 (2006), no.~1, 38--49.
\bibitem[D]{D} P. Deligne, Th\'{e}orie de Hodge II, Inst. Hautes \'{E}tudes Sci. Publ. Math. No. 40, (1971), 5--57.
\bibitem[EKL]{EKL} M.~Einsiedler, M.~Kapranov, D.~Lind, Non-Archimedean amoebas and tropical varieties, 
J. Reine Angew. Math.~601 (2006), 139--157.
\bibitem[EGA4]{EGA4} A.~Grothendieck, J.~Dieudonn\'e, \'El\'ements de G\'eom\'etrie Alg\'ebrique~IV: 
\'Etude locale des sch\'emas et des morphismes de sch\'emas, Inst. Hautes \'Etudes Sci. Publ. Math.~20 (1964), 24 (1965), 28 (1966), 32 (1967).
\bibitem[GeM]{GeM} I.~Gel'fand, R.~MacPherson, Geometry in Grassmannians and a generalization of the dilogarithm, Adv. in Math.~44 (1982), no.~3, 279--312.
\bibitem[GoM]{GoM} M.~Goresky, R.~MacPherson, Stratified Morse theory, 
Ergeb. Math. Grenzgeb. (3)~14, Springer (1988).
\bibitem[HKT1]{HKT1} P.~Hacking, S.~Keel, J.~Tevelev, Compactification of the moduli space of hyperplane arrangements, J. Algebraic Geom.~15 (2006), no.~4, 657--680.
\bibitem[HKT2]{HKT2} P.~Hacking, S.~Keel, J.~Tevelev, Stable pair, tropical, and log canonical compact moduli of del Pezzo surfaces,
preprint arXiv:math/0702505 [math.AG] (2007). 
\bibitem[H]{H} R.~Hartshorne, Algebraic geometry, Grad. Texts in Math.~52, Springer, 1977.
\bibitem[HM]{HM} R.~Hain, R.~MacPherson, Higher logarithms, Illinois J. Math.~34 (1990), no.~2, 392--475. 
\bibitem[KT]{KT} S.~Keel, J.~Tevelev, Geometry of Chow quotients of Grassmannians, Duke Math. J.~134 (2006), no.~2, 259--311.
\bibitem[L]{L} L.~Lafforgue, Chirurgie des grassmanniennes, CRM Monogr. Ser.~19, A.M.S. (2003).
\bibitem[MY]{MY} H.~Markwig, J.~Yu, Shellability of the moduli space of $n$ tropically collinear points in $\bR^d$, preprint arXiv:0711.0944v1 [math.CO] (2007).
\bibitem[S]{S} D.~Speyer, Uniformizing Tropical Curves I: Genus Zero and One, preprint arXiv:0711.2677v1 [math.AG](2007).
\bibitem[SS]{SS} D.~Speyer, B.~Sturmfels, The tropical Grassmannian, Adv. Geom.~4 (2004), no. 3, 389--411. 
\bibitem[ST]{ST} B.~Sturmfels, J.~Tevelev, Elimination Theory for Tropical Varieties, preprint
arXiv:0704.3471v1 [math.AG] (2007).
\bibitem[T]{T} J.~Tevelev, Compactifications of Subvarieties of Tori, Amer. J. Math~129 (2007), no. 4, 1087--1104. 
\bibitem[T2]{T2} J.~Tevelev, personal communication.
\bibitem[V]{V} C. Voisin, Hodge theory and complex algebraic geometry I, Cambridge Stud. Adv. Math. 76, C.U.P. (2002).


\end{thebibliography}
\end{document}